\documentclass[12pt]{amsart}

\usepackage[english]{babel}
\usepackage{mathrsfs,amssymb}
\usepackage{mathtools}
\usepackage[colorlinks, citecolor = blue]{hyperref}

\usepackage[shortlabels]{enumitem}
\setlist[itemize]{leftmargin=25pt}
\setlist[enumerate]{leftmargin=25pt}

\newtheorem{theorem}{Theorem}[section]
\newtheorem{lemma}[theorem]{Lemma}
\newtheorem{prop}[theorem]{Proposition}

\theoremstyle{definition}

\theoremstyle{remark}
\newtheorem{remark}[theorem]{Remark}

\numberwithin{equation}{section}
\usepackage[colorinlistoftodos,prependcaption,textsize=small]{todonotes}

\let \la=\lambda
\let \e=\varepsilon

\let \a=\alpha

\let \O=\Omega

\let \ga=\gamma

\allowdisplaybreaks

\begin{document}
\title[maximal operator on weighted Morrey spaces]
{A note on the maximal operator on weighted Morrey spaces}

\author[A.K. Lerner]{Andrei K. Lerner}
\address[A.K. Lerner]{Department of Mathematics,
Bar-Ilan University, 5290002 Ramat Gan, Israel}
\email{lernera@math.biu.ac.il}

\thanks{The author was supported by ISF grant no. 1035/21.}

\begin{abstract}
In this paper we consider weighted Morrey spaces ${\mathcal M}_{\la, {\mathcal F}}^p(w)$ adapted to a family of cubes ${\mathcal F}$, with norm
$$\|f\|_{{\mathcal M}_{\la, {\mathcal F}}^p(w)}:=\sup_{Q\in {\mathcal F}}\left(\frac{1}{|Q|^{\la}}\int_Q|f|^pw\right)^{1/p},$$
and the question we deal with is whether a Muckenhoupt-type condition characterizes the boundedness of the Hardy--Littlewood maximal operator on ${\mathcal M}_{\la, {\mathcal F}}^p(w)$.

In the case of the global Morrey spaces (when ${\mathcal F}$ is the family of all cubes in ${\mathbb R}^n$) this question is still open. In the case of the local Morrey spaces (when ${\mathcal F}$ is the family
of all cubes centered at the origin) this question was answered positively in a recent work of Duoandikoetxea--Rosenthal \cite{DR21}.

We obtain an extension of \cite{DR21} by showing that the answer is positive
when ${\mathcal F}$ is the family of all cubes centered at a sequence of points in ${\mathbb R}^n$ satisfying a certain lacunary-type condition.
\end{abstract}



\maketitle

\section{Introduction}
Denote by ${\mathcal Q}$ the family of all cubes in ${\mathbb R}^n$ with sides parallel to the axes. Given a cube $Q\in {\mathcal Q}$, denote $\langle f\rangle_Q:=\frac{1}{|Q|}\int_Qf$.
Recall that the Hardy--Littlewood maximal operator $M$ is defined by
$$Mf(x):=\sup_{Q\ni x}\langle |f|\rangle_Q,$$
where the supremum is taken over all cubes $Q\in {\mathcal Q}$ containing the point~$x$.

Let $X$ be a Banach function space over ${\mathbb R}^n$. We say that the $A_X$ condition holds if
there exists $C>0$ such that for every locally integrable $f$ and for all $Q\in {\mathcal Q}$,
$$\langle |f|\rangle_Q\|\chi_Q\|_{X}\le C\|f\chi_Q\|_{X}.$$
Using the language of the associate space $X'$, the $A_X$ condition can be written in the following equivalent form:
there exists $C>0$ such that for all $Q\in {\mathcal Q}$,
$$\frac{\|\chi_Q\|_{X}\|\chi_Q\|_{X'}}{|Q|}\le C.$$

The $A_X$ condition is a trivial necessary condition for the boundedness of the maximal operator $M$ on $X$. Indeed, this follows from the elementary fact that
$\langle |f|\rangle_Q\chi_Q\le M(f\chi_Q)$. Therefore, given a concrete space $X$, the question of interest is whether the $A_X$ condition is sufficient for the boundedness of $M$ on $X$.

Observe that this question usually is a challenging problem, and it does not have an universal answer, in general. For example, in the case where
$X$ is the weighted $L^p$ space, $X=L^p(w),p>1$, then the $A_{L^p(w)}$ condition is the classical Muckenhoupt's $A_p$ condition, and the answer to the above question
is positive \cite{M72}. However, if $X$ is a variable $L^p$ space, $X=L^{p(\cdot)}$, then the $A_{L^{p(\cdot)}}$ condition is not sufficient for the boundedness of $M$ on $L^{p(\cdot)}$, in general
(see, e.g., \cite[Th. 5.3.4]{DHHR11}).

In this paper we consider weighted Morrey spaces of Samko-type ${\mathcal M}_{\la}^p(w)$ with norm
\begin{equation}\label{ms}
\|f\|_{{\mathcal M}_{\la}^p(w)}:=\sup_{Q\in {\mathcal Q}}\left(\frac{1}{|Q|^{\la}}\int_Q|f|^pw\right)^{1/p},
\end{equation}
where $0<\la<1$ and $w$ is a non-negative locally integrable function. Observe that there are many equivalent notations for these spaces. For example, it is more common to use 
$\frac{1}{\ell_Q^{\la}}$ instead of $\frac{1}{|Q|^{\la}}$, where $\ell_Q$ is the side length of $Q$, and correspondingly to assume that $0<\la<n$ instead of $0<\la<1$. We hope that the above notation will not confuse the reader.

The space ${\mathcal M}_{\la}^p(w)$ is usually refereed to as the global Morrey space in contrast to the local Morrey space ${{\mathcal L}\mathcal M}_{\la}^p(w)$ where the supremum over all $Q\in {\mathcal Q}$ in (\ref{ms}) is replaced by the
supremum over all cubes $Q$ centered at the origin. More generally, one can consider Morrey spaces ${\mathcal M}_{\la, {\mathcal F}}^p(w)$ adapted to a given family of cubes ${\mathcal F}\subset {\mathcal Q}$,
with norm
$$\|f\|_{{\mathcal M}_{\la, {\mathcal F}}^p(w)}:=\sup_{Q\in {\mathcal F}}\left(\frac{1}{|Q|^{\la}}\int_Q|f|^pw\right)^{1/p}.$$

To the best of our knowledge, the question whether the $A_{{\mathcal M}_{\la}^p(w)}$ condition is sufficient for the boundedness of $M$ on ${\mathcal M}_{\la}^p(w)$ is still open.
On the other hand, for the local Morrey spaces it has been recently established by Duoandikoetxea--Rosenthal \cite{DR21}, simplifying the previous characterization by Nakamura--Sawano--Tanaka \cite{NST20},
that the $A_{{{\mathcal L}\mathcal M}_{\la}^p(w)}$ condition is necessary and sufficient for the boundedness of~$M$ on ${{\mathcal L}\mathcal M}_{\la}^p(w)$.

Our main result is an extension of the Duoandikoetxea--Rosenthal result to more general families of cubes ${\mathcal F}$, namely, we will prove the following.

\begin{theorem}\label{MR} Suppose that $\nu>1$. Let $\Lambda:=\{x_j\}$ be a sequence of points in ${\mathbb R}^n$ such that
\begin{equation}\label{rcond}
\max\big(|x_i|,|x_j|\big)\le \nu|x_i-x_j|\quad(i\not=j).
\end{equation}
Denote by ${\mathcal F}$ the family of all cubes centered at points from $\Lambda$. Then the maximal operator $M$ is bounded on
${\mathcal M}_{\la, {\mathcal F}}^p(w), p>1,$ if and only if the
$A_{{\mathcal M}_{\la, {\mathcal F}}^p(w)}$ condition holds.
\end{theorem}

\begin{remark}\label{ex}
In the one-dimensional case, a typical example of the sequence satisfying (\ref{rcond}) is the lacunary sequence $\{0,\pm \ga^j\}_{j\in {\mathbb Z}},$ where $\ga:=\frac{\nu}{\nu-1}$.
If $n\ge 2$, for every $j\in {\mathbb Z}$, there exists a finite number $N=N(\nu,n)$ of points $\xi_{l,j}$ on the sphere $\{y:|y|=\ga^j\}$
such that
$$\ga^j\le \nu|\xi_{l,j}-\xi_{m,j}|\quad(l\not=m).$$
Then the sequence $\{0,\xi_{l,j}\}_{j\in {\mathbb Z},\, l=1,\dots, N(\nu,n)}$ will satisfy (\ref{rcond}).
\end{remark}

The proof in Duoandikoetxea--Rosenthal \cite{DR21} is based essentially on the notion of local $A_p$ weights developed in \cite{LS10}. It seems that this
approach has some limitations, and, in particular, it is not clear how it can be directly extended to more
general families of cubes. We use instead the notion of dyadic $A_p$ weights, which seems to be more flexible. Also our key ingredient is the following.

\begin{theorem}\label{mr}
Let $\Omega\subset {\mathbb R}^n$ be an arbitrary nonempty closed set.
Given $0<r_1<r_2<\infty$, define
\begin{equation}\label{W}
{\mathcal W}_{r_1,r_2}:=\{Q\in{\mathcal Q}:r_1\,{\rm{diam}}\,Q\le {\rm{dist}}(Q,\O)\le r_2\,{\rm{diam}}\,Q\}.
\end{equation}
Then the maximal operator $M$ is bounded on $M_{\la,{\mathcal W}_{r_1,r_2}}^p(w),p>1,$ if and only if the $A_{M_{\la,{\mathcal W}_{r_1,r_2}}^p(w)}$ condition holds.
\end{theorem}

Observe that the standard example of the family ${\mathcal W}_{r_1,r_2}$ in Theorem~\ref{mr} is the Whitney covering of $\Omega^c$. Also note that the implicit constants
in this statement depend only on $r_1$ and $r_2$ and do not depend on $\Omega$.

The main connecting element between Theorems \ref{MR} and \ref{mr} is the following lemma, which is based on the same idea as in \cite[Lemma 1.1]{NST20}.

\begin{lemma}\label{connect}
Assume that $\Lambda$ and ${\mathcal F}$ are the same as in Theorem \ref{MR}. Let $\Omega:=\{x_j\in \Lambda\}$. Then there exist $r_1$ and $r_2$ (depending on $\nu$ from (\ref{rcond}))
such that for the family ${\mathcal W}_{r_1,r_2}$ defined by (\ref{W}),
$$\|f\|_{M_{\la,{\mathcal F}}^p(w)}\simeq \|f\|_{M_{\la, {\mathcal W}_{r_1,r_2}}^p(w)}.$$
\end{lemma}

It is easy to see that Lemma \ref{connect} and Theorem \ref{mr} immediately imply Theorem \ref{MR}. Indeed, since the $A_X$ condition is defined in terms on the $X$-norm only, by Lemma \ref{connect}
we have that the $A_{M_{\la,{\mathcal F}}^p(w)}$ and $A_{M_{\la, {\mathcal W}_{r_1,r_2}}^p(w)}$ conditions are equivalent. Also by Lemma \ref{connect}, $M$ is bounded on $M_{\la,{\mathcal F}}^p(w)$ iff $M$
is bounded on $M_{\la, {\mathcal W}_{r_1,r_2}}^p(w)$. Hence, in order to obtain Theorem \ref{MR}, it remains to apply Theorem \ref{mr}.

The paper is organized as follows. Section 2 contains brief preliminaries about dyadic lattices and dyadic $A_p$ weights. In Section 3 we prove Theorem \ref{mr} and Lemma \ref{connect}.

\section{Preliminaries}
\subsection{Dyadic lattices}
Given a cube $Q_0\in {\mathcal Q}$, let ${\mathcal D}(Q_0)$ denote the set of all dyadic cubes with respect to $Q_0$, that is, the cubes
obtained by repeated subdivision of $Q_0$ and each of its descendants into $2^n$ congruent subcubes.

Following \cite{LN19} we say that a dyadic lattice ${\mathscr D}$ in ${\mathbb R}^n$ is any collection of cubes such that
\begin{enumerate}
\renewcommand{\labelenumi}{(\roman{enumi})}
\item
if $Q\in{\mathscr D}$, then each child of $Q$ is in ${\mathscr D}$ as well;
\item
every 2 cubes $Q',Q''\in {\mathscr D}$ have a common ancestor, i.e., there exists $Q\in{\mathscr D}$ such that $Q',Q''\in {\mathcal D}(Q)$;
\item
for every compact set $K\subset {\mathbb R}^n$, there exists a cube $Q\in {\mathscr D}$ containing $K$.
\end{enumerate}

Given a dyadic lattice ${\mathscr D}$, the dyadic maximal operator $M^{\mathscr D}$ is defined by
$$M^{\mathscr D}f(x):=\sup_{Q\in {\mathscr D}: x\in Q}\langle|f|\rangle_Q.$$

It is well-known that there exist $3^n$ dyadic lattices ${\mathscr D}_j$ such that for all $x\in {\mathbb R}^n$,
\begin{equation}\label{three}
Mf(x)\le 3^n\sum_{j=1}^{3^n}M^{{\mathscr D}_j}f(x).
\end{equation}
This follows from the fact that every cube in ${\mathbb R}^n$ can be well approximated by the cubes from ${\mathscr D}_j$ (see, e.g., \cite[Lemma 2.5]{HLP13} or \cite[Th. 3.1]{LN19}).

\subsection{Dyadic $A_p$ weights}
Fix a dyadic lattice ${\mathscr D}$. Assume that ${\mathcal F}\subset~{\mathscr D}$. We say that a weight $w\in A_{p, {\mathcal F}}$ if
$$[w]_{A_{p, {\mathcal F}}}:=\sup_{Q\in {\mathcal F}}\left(\frac{1}{|Q|}\int_Qw\right)\left(\frac{1}{|Q|}\int_Qw^{-\frac{1}{p-1}}\right)^{p-1}<\infty.$$

Next we consider the dyadic maximal operator $M^{\mathcal F}$ which is defined similarly to $M^{\mathscr D}$ but with the supremum restricted only to the cubes $Q\in {\mathcal F}$.
The following statement can be easily deduced from a general result of Jawerth \cite{J86}, for a simpler proof see \cite{L08}.

\begin{prop}\label{dmop} The dyadic maximal operator $M^{\mathcal F}$ is bounded on $L^p(w), p>1,$ if and only if $w\in A_{p, {\mathcal F}}$.
\end{prop}

\section{Proofs of Theorem \ref{mr} and Lemma \ref{connect}}
We start with the following auxiliary statement. Recall that given a nonempty compact set $\O\subset {\mathbb R}^n$, the family ${\mathcal W}_{r_1,r_2}$ is defined by
$$
{\mathcal W}_{r_1,r_2}:=\{Q\in{\mathcal Q}:r_1\,{\rm{diam}}\,Q\le {\rm{dist}}(Q,\O)\le r_2\,{\rm{diam}}\,Q\}.
$$

\begin{lemma}\label{eqst} Given $0<r_1<r_2<\infty$, there exist $1<\a_1<\a_2<\infty$ such that
$$\|f\|_{M^p_{\la,{\mathcal W}_{r_1,r_2}(w)}}\simeq \|f\|_{M^p_{\la,{\mathcal W}_{\a_1,\a_2}(w)}}.$$
\end{lemma}

\begin{proof}
One can assume that $r_1\le 1$, otherwise there is nothing to prove.
Let $Q\in {\mathcal W}_{r_1,r_2}$. Split $Q$ into $2^n$ congruent subcubes $Q_j$. Then we have
$$2r_1\text{diam}\,Q_j=r_1\text{diam}\,Q\le \text{dist}(Q,\O)\le \text{dist}\,(Q_j,\O).$$
On the other hand,
\begin{eqnarray*}
\text{dist}\,(Q_j,\O)&\le& \text{diam}\,Q+\text{dist}(Q,\O)\\
&\le& (r_2+1)\text{diam}\,Q=2(r_2+1)\text{diam}\,Q_j.
\end{eqnarray*}
Therefore,
\begin{eqnarray*}
\frac{1}{|Q|^{\la}}\int_Q|f|^pw&\le& \sum_{j=1}^{2^n}\frac{2^{-n\la}}{|Q_j|^{\la}}\int_{Q_j}|f|^pw\\
&\le& 2^{n(1-\la)}\|f\|^p_{M^p_{\la,{\mathcal W}_{2r_1,2(r_2+1)}(w)}},
\end{eqnarray*}
which implies
$$\|f\|_{M^p_{\la,{\mathcal W}_{r_1,r_2}(w)}}\le 2^{n(1-\la)/p}\|f\|_{M^p_{\la,{\mathcal W}_{2r_1,2(r_2+1)}(w)}}.$$
Iterating this estimate, we obtain that there exist $1<\a_1<\a_2<\infty$ such that
$$\|f\|_{M^p_{\la,{\mathcal W}_{r_1,r_2}(w)}}\lesssim \|f\|_{M^p_{\la,{\mathcal W}_{\a_1,\a_2}(w)}}$$
with the implicit constant depending on $r_1,\la, p$ and $n$. Observe that $\a_1$  can be taken in the interval $(1,2]$ and $\a_2$ will
depend on $r_1$ and~$r_2$.

It remains to prove the converse estimate. Suppose that $r_1<\a_1$ and $Q\in {\mathcal W}_{\a_1,\a_2}$. Then $\a_1\le \frac{\text{dist}(Q,\O)}{\text{diam}\,Q}$.
Since $\lim_{|Q|\to \infty}\frac{\text{dist}(Q,\O)}{\text{diam}\,Q}=0$, by continuity, there exists a cube $\widetilde Q\supset Q$ such that
$\frac{\text{dist}(\widetilde Q,\O)}{\text{diam}\,\widetilde Q}=r_1$. Therefore, $\widetilde Q\in {\mathcal W}_{r_1,r_2}$. Also,
$$\text{diam}\,\widetilde Q=\frac{1}{r_1}\text{dist}(\widetilde Q,\O)\le \frac{1}{r_1}\text{dist}(Q,\O)\le \frac{\a_2}{r_1}\text{diam}\, Q.$$
From this,
$$\frac{1}{|Q|^{\la}}\int_Q|f|^pw\le \Big(\frac{\a_2}{r_1}\Big)^{\la n}\frac{1}{|\widetilde Q|^{\la}}\int_{\widetilde Q}|f|^pw,$$
which implies
$$\|f\|_{M^p_{\la,{\mathcal W}_{\a_1,\a_2}(w)}}\le (\a_2/r_1)^{\la n/p}\|f\|_{M^p_{\la,{\mathcal W}_{r_1,r_2}(w)}},$$
and therefore, the proof is complete.
\end{proof}

\begin{proof}[Proof of Theorem \ref{mr}]
As discussed in the Introduction, it is enough to prove the sufficiency of the $A_{M_{\la,{\mathcal W}_{r_1,r_2}}^p(w)}$ condition for the boundedness of $M$ on $M_{\la,{\mathcal W}_{r_1,r_2}}^p(w)$.
Further, by (\ref{three}), it suffices to prove the boundedness of the dyadic maximal operator
$M^{\mathscr D}$ for a fixed dyadic lattice~${\mathscr D}$.
Also, by Lemma \ref{eqst}, without loss of generality one can assume that $1<r_1<r_2<\infty$.

Let us split the cubes in ${\mathscr D}$ into two families
$${\mathscr G}_1:=\{R\in {\mathscr D}: \text{dist}(R,\Omega)\le \a\,\text{diam}\,R\},$$
$${\mathscr G}_2:=\{R\in {\mathscr D}: \text{dist}(R,\Omega)>\a\,\text{diam}\,R\},$$
where $\a>0$ is chosen in such a way that $\mu:=2\frac{r_2+1}{\a-1}+1$ will satisfy $\mu<\min(3/2,r_1)$, and, moreover, $r_1<\a$.

Split correspondingly the maximal operator
$$M^{\mathscr D}f=\max(M^{{\mathscr G}_1}f, M^{{\mathscr G}_2}f).$$

Fix a cube $Q\in {\mathcal W}_{r_1,r_2}$, and consider $M^{{\mathscr G}_1}f$ on $Q$.
Suppose that $R\in {\mathscr G}_1$ and $R\cap Q\not=\emptyset$. For every $x\in Q, y\in R$ and $\xi\in \O$,
$$|x-\xi|\le |x-y|+|y-\xi|\le \text{diam}\,Q+\text{diam}\,R+|y-\xi|.$$
From this,
\begin{equation}\label{QR}
\text{dist}(Q,\Omega)\le \text{diam}\,Q+\text{diam}\,R+\text{dist}(R,\Omega),
\end{equation}
and hence
$$\text{diam}\,Q\le\frac{\a+1}{r_1-1}\text{diam}\,R.$$
Therefore, $Q\subset \ga R,$ where $\ga:=2\frac{\a+1}{r_1-1}+1$.

Further,
$$\|\chi_{\ga R}\|_{M_{\la,{\mathcal W}_{r_1,r_2}}^p(w)}\ge \|\chi_{Q}\|_{M_{\la,{\mathcal W}_{r_1,r_2}}^p(w)}\ge \left(\frac{w(Q)}{|Q|^{\la}}\right)^{1/p},$$
which implies
\begin{eqnarray*}
\langle |f|\rangle_R\le \ga^n \langle |f|\rangle_{\ga R}&\le& \frac{\ga^n}{\|\chi_{\ga R}\|_{M_{\la,{\mathcal W}_{r_1,r_2}}^p(w)}}\|f\|_{M_{\la,{\mathcal W}_{r_1,r_2}}^p(w)}\\
&\le& \ga^n\left(\frac{|Q|^{\la}}{w(Q)}\right)^{1/p}\|f\|_{M_{\la,{\mathcal W}_{r_1,r_2}}^p(w)}.
\end{eqnarray*}
Thus, we obtain
\begin{equation}\label{1part}
\|M^{{\mathscr G}_1}f\|_{L^{\infty}(Q)}\le \ga^n\left(\frac{|Q|^{\la}}{w(Q)}\right)^{1/p}\|f\|_{M_{\la,{\mathcal W}_{r_1,r_2}}^p(w)}.
\end{equation}

Turn to the second part, namely consider $\int_Q(M^{{\mathscr G}_2}f)^pw$. Denote
$${\mathscr G}_2(Q):=\{R\in {\mathscr G}_2:R\cap Q\not=\emptyset\}.$$

Let us show that $w\in A_{p,{\mathscr G}_2(Q)}$ or, equivalently, that for any locally integrable $f$,
\begin{equation}\label{apeq}
\langle |f|\rangle_R\lesssim \left(\frac{1}{w(R)}\int_R|f|^pw\right)^{1/p}
\end{equation}
uniformly in $R\in {\mathscr G}_2(Q)$.

Fix an $R\in {\mathscr G}_2(Q)$. Let us consider
$$\|f\chi_R\|_{M_{\la, {\mathcal W}_{r_1,r_2}}^p(w)}=\sup_{L\in {\mathcal W}_{r_1,r_2}}\left(\frac{1}{|L|^{\la}}\int_{L\cap R}|f|^pw\right)^{1/p}.$$
First, observe that, similarly to (\ref{QR}),
$$
\text{dist}(R,\Omega)\le \text{diam}\,R+\text{diam}\,Q+\text{dist}(Q,\Omega),
$$
which implies $\text{diam}\,R\le\frac{r_2+1}{\a-1}\text{diam}\,Q$, and therefore $R\subset \mu Q$.
From this, for every $L\in {\mathcal W}_{r_1,r_2}$ such that $L\cap R\not=\emptyset$,
$$\text{dist}(Q,\Omega)\le \mu\,\text{diam}\,Q+\text{diam}\,L+\text{dist}(L,\Omega),$$
and hence
$$\text{diam}\,Q\le\frac{r_2+1}{r_1-\mu}\text{diam}\,L.$$
Therefore,
\begin{equation}\label{inst}
\|f\chi_R\|_{M_{\la, {\mathcal W}_{r_1,r_2}}^p(w)}\lesssim\left(\frac{1}{|Q|^{\la}}\int_{R}|f|^pw\right)^{1/p}.
\end{equation}

Next, let us consider $\|\chi_R\|_{M_{\la, {\mathcal W}_{r_1,r_2}}^p(w)}$. Since
$$r_1<\a<\frac{\text{dist}(R,\O)}{\text{diam}\,R}$$
and $\frac{\text{dist}(R,\O)}{\text{diam}\,R}\to 0$ as $R$ expands, by continuity there exists a cube $R'$ such that $R\subset R'$ and $\frac{\text{dist}(R',\O)}{\text{diam}\,R'}=r_1$.
Hence, $R'\in {\mathcal W}_{r_1,r_2}$. Also, since $R'\cap Q\not=\emptyset$,
$$r_1\text{diam}\,R'=\text{dist}(R',\Omega)\le \text{diam}\,R'+\text{diam}\,Q+\text{dist}(Q,\O),$$
and hence $\text{diam}\,R'\le\frac{r_2+1}{r_1-1}\text{diam}\,Q$. Therefore,
\begin{equation}\label{sint}
\|\chi_R\|_{M_{\la, {\mathcal W}_{r_1,r_2}}^p(w)}\ge \left(\frac{w(R)}{|R'|^{\la}}\right)^{1/p}\gtrsim \left(\frac{w(R)}{|Q|^{\la}}\right)^{1/p}.
\end{equation}

Combining (\ref{inst}) with (\ref{sint}) and using the $A_{M_{\la,{\mathcal W}_{r_1,r_2}}^p(w)}$ condition, we obtain
$$\langle|f|\rangle_R\left(\frac{w(R)}{|Q|^{\la}}\right)^{1/p}\lesssim \langle|f|\rangle_R\|\chi_R\|_{M_{\la, {\mathcal F}}^p(w)}\lesssim \left(\frac{1}{|Q|^{\la}}\int_{R}|f|^pw\right)^{1/p},$$
which proves (\ref{apeq}). Therefore, by Proposition \ref{dmop},
$$\int_{Q}(M^{{\mathscr G}_2}f)^pw\le \int_{Q}(M^{{\mathscr G}_2(Q)}(f\chi_{\mu Q}))^pw\lesssim \int_{\mu Q}|f|^pw.$$
This estimate along with (\ref{1part}) implies
\begin{eqnarray*}
\left(\frac{1}{|Q|^{\la}}\int_Q(M^{\mathscr D}f)^pw\right)^{1/p}&\le&\left(\frac{1}{|Q|^{\la}}\int_Q(M^{{\mathscr G}_1}f)^pw\right)^{1/p}\\
&+&\left(\frac{1}{|Q|^{\la}}\int_Q(M^{{\mathscr G}_2}f)^pw\right)^{1/p}\\
&\lesssim& \|f\|_{M_{\la,{\mathcal W}_{r_1,r_2}}^p(w)}+\left(\frac{1}{|Q|^{\la}}\int_{\mu Q}|f|^pw\right)^{1/p}.
\end{eqnarray*}

In order to complete the proof, it remains to show that
\begin{equation}\label{itrem}
\left(\frac{1}{|Q|^{\la}}\int_{\mu Q}|f|^pw\right)^{1/p}\lesssim \|f\|_{M_{\la,{\mathcal W}_{r_1,r_2}}^p(w)}.
\end{equation}
First, observe that
\begin{equation}\label{obser}
\text{dist}(Q,\O)\le \frac{\mu-1}{2}\text{diam}\,Q+\text{dist}(\mu Q,\Omega).
\end{equation}
Indeed, the trivial estimates
$$r_1\text{diam}\,Q\le\text{dist}(Q,\Omega)\le \mu\,\text{diam}\,Q+\text{dist}(\mu Q, \Omega)$$
along with $\mu<r_1$ show that $\text{dist}(\mu Q, \Omega)>0$. Therefore, there exist $x_0\in \partial(\mu Q)$ and $\xi_0\in \Omega$ such that
$$\text{dist}(\mu Q, \Omega)=|x_0-\xi_0|.$$
Now take $x'\in \partial Q$ closest to $x_0$. Then using that $\text{dist}(Q,\Omega)\le |x'-\xi_0|$ and $|x'-x_0|\le \frac{\mu-1}{2}\text{diam}\,Q$ we obtain (\ref{obser}).

Divide $\mu Q$ into $2^n$ equal subcubes $P_j, j=1,\dots 2^n$. By (\ref{obser}), for every $j$,
\begin{eqnarray*}
\frac{2}{\mu}\Big(r_1-\frac{\mu-1}{2}\Big)\text{diam}\,P_j&=&\Big(r_1-\frac{\mu-1}{2}\Big)\text{diam}\,Q\\
&\le& \text{dist}(\mu Q,\O)\le \text{dist}(P_j,\Omega).
\end{eqnarray*}

Further, note that $r_1<\frac{2}{\mu}\Big(r_1-\frac{\mu-1}{2}\Big)$. Indeed, this is equivalent to $\frac{\mu-1}{2-\mu}<r_1$, and using that $\mu<3/2$ we have
$\frac{\mu-1}{2-\mu}<1<r_1$. Therefore, the cubes $P_j$ satisfy
$$r_1\text{diam}\,P_j<\text{dist}(P_j,\O).$$
Similarly to the argument we have used to establish (\ref{sint}), for every $j$, there exists the cube $\widetilde P_j\in {\mathcal W}_{r_1,r_2}$ such that $P_j\subset \widetilde P_j$ and
$\text{diam}\widetilde P_j\lesssim \text{diam}\,Q$.
Hence,
$$\frac{1}{|Q|^{\la}}\int_{\nu Q}|f|^pw\lesssim \sum_{j=1}^{2^n}\frac{1}{|\widetilde P_j|^{\la/n}}\int_{\widetilde P_j}|f|^pw\lesssim \|f\|_{M_{\la,{\mathcal W}_{r_1,r_2}}^p(w)}^p,$$
which proves (\ref{itrem}), and therefore, the proof is complete.
\end{proof}

Given a cube $Q$, denote by $c_Q$ and $\ell_Q$ the center and the side length of $Q$, respectively. The following lemma is close to \cite[Lemma 1.1]{NST20}.

\begin{lemma}\label{redw}
For every cube $Q\in {\mathcal Q}$ and $N\in {\mathbb N}$,
$$\frac{1}{|Q|^{\la}}\int_Q|f|^pw\lesssim \sup_{R\in {\mathcal F}_N}\frac{1}{|R|^{\la}}\int_R|f|^pw,$$
where
$${\mathcal F}_N=\{R\subset Q:\frac{N}{\sqrt n}\,{\rm{diam}}\,R\le {\rm{dist}}(R,c_Q)\le N{\rm{diam}}\,R\}$$
and the implicit constant depends on $\la, N$ and $n$.
\end{lemma}

\begin{proof} Given a cube $P\in {\mathcal Q}$ and $r>1$, consider the cubic annulus $rP\setminus P$.
Note that $|rP\setminus P|=(r^n-1)|P|$. Let $L$ be a cube of the maximal diameter such that $L\subset rP\setminus P$. Then
$|L|=\Big(\frac{r-1}{2}\Big)^n|P|$ and
$$\frac{1}{{\sqrt n}(r-1)}\text{diam}\,L=\frac{1}{2}\ell_P\le \text{dist}(L,c_P)\le \frac{1}{2}\text{diam}\,P=\frac{1}{r-1}\text{diam}\,L.$$

Taking here $r=1+\frac{1}{N}$, where $N\in {\mathbb N}$, we obtain that $(1+1/N)P\setminus P$ can be covered by $2^n((N+1)^n-N^n)$ pairwise disjoint cubes $L_j$
such that $|L_j|=|P|/(2N)^n$ and
\begin{equation}\label{Lj}
\frac{N}{\sqrt n}\text{diam}\,L_j\le \text{dist}(L_j,c_P)\le N\text{diam}\,L_j.
\end{equation}

We now apply this argument to the cubes $\ga^{k+1}Q$, where $\ga:=\frac{N}{N+1}$ and $k\in {\mathbb Z}_+$. We obtain that $\ga^kQ\setminus \ga^{k+1}Q=\cup_{j=1}^{2^n((N+1)^n-N^n)}L_{j,k}$,
where the cubes $L_{j,k}$ satisfy the same estimates as the cubes $L_j$ in (\ref{Lj}).
Therefore,
\begin{eqnarray*}
&&\frac{1}{|Q|^{\la}}\int_Q|f|^pw=\sum_{k=0}^{\infty}\frac{1}{|Q|^{\la}}\int_{\ga^kQ\setminus \ga^{k+1}Q}|f|^pw\\
&&=\frac{1}{(2(N+1))^{\la n}}\sum_{k=0}^{\infty}\Big(\frac{N}{N+1}\Big)^{k\la n}\sum_{j=1}^{2^n((N+1)^n-N^n)}\frac{1}{|L_{j,k}|^{\la}}\int_{L_{j,k}}|f|^pw\\
&&\lesssim \sup_{R\in {\mathcal F}_N}\frac{1}{|R|^{\la}}\int_R|f|^pw,
\end{eqnarray*}
which completes the proof.
\end{proof}

\begin{proof}[Proof of Lemma \ref{connect}]
First, it is easy to see that for every $0<r_1<r_2<\infty$,
\begin{equation}\label{ets}
\|f\|_{M_{\la, {\mathcal W}_{r_1,r_2}}^p(w)}\lesssim \|f\|_{M_{\la,{\mathcal F}}^p(w)}.
\end{equation}
Indeed, suppose that $Q\in {\mathcal W}_{r_1,r_2}$. This means that there exists $x_j\in \Lambda$ such that
$$r_1\text{diam}\,Q\le \text{dist}(Q,x_j)\le r_2\text{diam}\,Q.$$
From this, there exists a cube $Q'\supset Q$ centered at $x_j$ and such that $|Q'|\lesssim |Q|$. Hence,
$$\frac{1}{|Q|^{\la}}\int_Q|f|^pw\lesssim \frac{1}{|Q'|^{\la}}\int_{Q'}|f|^pw,$$
which proves (\ref{ets}).

Let us show now that there exist $0<r_1<r_2<\infty$ such that
\begin{equation}\label{md}
\|f\|_{M_{\la,{\mathcal F}}^p(w)}\lesssim \|f\|_{M_{\la, {\mathcal W}_{r_1,r_2}}^p(w)}.
\end{equation}

Without loss of generality we can assume that $0\in \Lambda$ (if $0\not\in \Lambda$ the proof is essentially the same).

For every $x_j\in \Lambda, x_j\not=0$ denote by $Q_{\nu, x_j}$ the cube centered at $x_j$ with $\text{diam}\,Q_{\nu,x_j}=\frac{1}{\nu}|x_j|$.
The key property of these cubes is that
\begin{equation}\label{kp}
R\subset Q_{\nu, x_j}\Rightarrow \text{dist}(R,\O)=\text{dist}(R,x_j).
\end{equation}
Indeed, on the one hand,
$$\text{dist}(R,x_j)\le \frac{1}{2}\text{diam}\,Q_{\nu,x_j}=\frac{1}{2\nu}|x_j|.$$
On the other hand, by (\ref{rcond}), for any $x_i\in {\Lambda}$ such that $x_i\not=x_j$,
\begin{eqnarray*}
\text{dist}(R,x_i)\ge \text{dist}(Q_{\nu,x_j},x_i)&\ge& |x_i-x_j|-\frac{1}{2}\text{diam}\,Q_{\nu,x_j}\\
&\ge& \frac{1}{\nu}|x_j|-\frac{1}{2\nu}|x_j|=\frac{1}{2\nu}|x_j|,
\end{eqnarray*}
which proves (\ref{kp}).

Let $N\in {\mathbb N}$ which will be specified later. Consider the families
$$W_j:=\{R\subset Q_{\nu,x_j}:\frac{N}{\sqrt n}\,{\rm{diam}}\,R\le {\rm{dist}}(R,x_j)\le N{\rm{diam}}\,R\}.$$
Then, by (\ref{kp}),
\begin{equation}\label{emb}
W_j\subset {\mathcal W}_{\frac{N}{\sqrt n}, N}.
\end{equation}
Also denote
$$W:=\{R\in {\mathcal Q}:\frac{N}{\sqrt n}\,{\rm{diam}}\,R\le {\rm{dist}}(R,0)\le N{\rm{diam}}\,R\}.$$

Let $Q\in {\mathcal F}$ be an arbitrary cube. If $Q$ is centered at the origin, then, by Lemma \ref{redw},
\begin{equation}\label{zero}
\frac{1}{|Q|^{\la}}\int_Q|f|^pw\lesssim \|f\|^p_{M^p_{\la, W}(w)}.
\end{equation}
Suppose that $Q$ is centered at $x_j\not=0$. If $Q\supset Q_{\nu, x_j}$, then there exists a cube $\widetilde Q\supset Q$ centered at the origin and such that $|\widetilde Q|\lesssim |Q|$. In this case we again apply (\ref{zero}).
If $Q\subset Q_{\nu,x_j}$, then, by Lemma \ref{redw} and by (\ref{emb}),
$$
\frac{1}{|Q|^{\la}}\int_Q|f|^pw\lesssim \|f\|^p_{M^p_{\la, W_j}(w)}\le \|f\|^p_{M^p_{\la, {\mathcal W}_{\frac{N}{\sqrt n}, N}}(w)}.
$$
Therefore,
\begin{equation}\label{unes}
\|f\|_{M_{\la,{\mathcal F}}^p(w)}\lesssim \max\Big(\|f\|_{M^p_{\la, W}(w)}, \|f\|_{M^p_{\la, {\mathcal W}_{\frac{N}{\sqrt n}, N}}(w)}\Big).
\end{equation}

It remains to estimate $\|f\|_{M^p_{\la, W}(w)}$. Fix a cube $Q\in W$. There exists $x_j\in \Lambda$ such that
$$\text{dist}(Q,\O)=\text{dist}(Q,x_j).$$
If $x_j=0$, then $Q\in {\mathcal W}_{\frac{N}{\sqrt n}, N}$.

Suppose that $x_j\not=0$. Let $0<\e<1$, which will be specified later. Assume that
$\text{dist}(Q,x_j)\ge \e\,\text{dist}(Q,0)$. In this case $Q\in {\mathcal W}_{\frac{\e N}{\sqrt n}, N}$.
Observe that the cases considered so far yield
\begin{equation}\label{conssf}
\frac{1}{|Q|^{\la}}\int_Q|f|^pw\le \|f\|^p_{M^p_{\la, {\mathcal W}_{\frac{\e N}{\sqrt n}, N}}(w)}.
\end{equation}

It remains to consider the case where $\text{dist}(Q,x_j)<\e\,\text{dist}(Q,0)$.
Then there exists $\xi\in Q$ such that
$$|\xi-x_j|<\e\,\text{dist}(Q,0)<\e N\text{diam}\,Q.$$
On the other hand,
$$\frac{N}{\sqrt n}\text{diam}\,Q\le \text{dist}(Q,0)\le |\xi|.$$
Therefore,
\begin{equation}\label{lest}
N\Big(\frac{1}{\sqrt n}-\e\Big)\text{diam}\,Q\le |\xi|-|\xi-x_j|\le |x_j|.
\end{equation}

%We also have
%$$|\xi|\le \text{diam}\,Q+\text{dist}(Q,0)\le (N+1)\text{diam}\,Q.$$
%Hence,
%\begin{equation}\label{upes}
%|x_j|\le |x_j-\xi|+|\xi|\le (\e N+N+1)\text{diam}\,Q.
%\end{equation}

Let $\widetilde Q$ be the minimal cube centered at $x_j$ which contains $Q$. Then
\begin{equation}\label{ffes}
\text{diam}\,\widetilde Q\le 2(\text{diam}\,Q+\text{dist}(Q,x_j))\le 2(1+\e N)\text{diam}\,Q.
\end{equation}
From this, by (\ref{lest}),
$$\text{diam}\,\widetilde Q\le \frac{2{\sqrt n}(1+\e N)}{N(1-\e{\sqrt n})}|x_j|.$$
Choose now $\e$ and $N$ in such a way to have
\begin{equation}\label{equa}
\frac{2{\sqrt n}(1+\e N)}{N(1-\e{\sqrt n})}=\frac{1}{\nu}.
\end{equation}
Then, applying (\ref{ffes}), Lemma \ref{redw} and (\ref{emb}), we obtain
$$
\frac{1}{|Q|^{\la}}\int_Q|f|^pw\lesssim \frac{1}{|\widetilde Q|^{\la}}\int_{\widetilde Q}|f|^pw\lesssim \|f\|^p_{M^p_{\la,W_j}(w)}\le
\|f\|^p_{M^p_{\la, {\mathcal W}_{\frac{N}{\sqrt n}, N}}(w)},
$$
which, combined with (\ref{conssf}), yields
$$
\|f\|_{M^p_{\la, W}(w)}\lesssim \|f\|_{M^p_{\la,{\mathcal W}_{\frac{\e N}{\sqrt n}, N}}(w)},
$$
where $0<\e<1$ and $N\in {\mathbb N}$ satisfy (\ref{equa}). This, combined with (\ref{unes}), completes the proof.
\end{proof}

\end{document}